\documentclass{amsart}

\usepackage{amsthm,amstext,amscd,amsbsy,amsxtra,latexsym}
\usepackage[english]{babel}
\usepackage[latin1]{inputenc}
\usepackage{verbatim} 
  
\usepackage[msc-links]{amsrefs}

\usepackage[capitalise]{cleveref}
\crefname{thm}{Thm.}{}
\crefname{prop}{Prop.}{}
\crefname{lem}{Lem.}{}
\crefname{cor}{Cor.}{}
\crefname{prob}{Problem}{}
\crefname{figure}{Fig.}{}

\newtheorem{thm}{Theorem}

\newtheorem {prop}{Proposition}
\newtheorem {lem}{Lemma}
\newtheorem {defi}{Definition}
\newtheorem {exa}{Example}

\newtheorem {cor}{Corollary}

\newcommand\N{\mathbb N}
\newcommand\Z{\mathbb Z}
\newcommand\Q{\mathbb Q}
\newcommand\R{\mathbb R}

\def\P{\mathbb P}

\def\wP{\mathbb{WP}}

\newcommand\A{\mathcal A}
\newcommand\X{\mathcal X}            

\newcommand\w{\mathfrak w}

\newcommand\iso{\cong}

\DeclareMathOperator\wh{\mathfrak{h}}   
\DeclareMathOperator\awh{\mathfrak{\tilde h} }   

\DeclareMathOperator\Proj{Proj}
\DeclareMathOperator\Spec{Spec}

\DeclareMathOperator\wt{wt}

\newcommand\<{\langle}

\def\a{\alpha}

\def\deg{\mbox{deg }}

\def\iso{\equiv}

\def\p{\mathfrak p}
\def\O{\mathcal O}

\DeclareMathOperator\wgcd{wgcd \, }
\DeclareMathOperator\awgcd{\overline{wgcd}\, }

\DeclareMathOperator\hwgcd{h_{wgcd} \, }
\DeclareMathOperator\hawgcd{h_{\overline{wgcd}}\, }

\def\J{\mathfrak J}

\def\l{\lambda}
\def\A{\mathbb A}

\def\x{\mathbf x}
\def\y{\mathbf y}

\begin{document}

\title[Weighted gcd and weighted heights]{Weighted greatest common divisors and weighted heights}

\author{L. Beshaj,  J. Gutierrez,   T. Shaska}

\begin{abstract}
We introduce the weighted greatest common divisor of a tuple of integers and explore some of its basic properties. Furthermore, for a set of heights $\w=(q_0, \ldots  , q_n)$,  we use the concept of the weighted greatest common divisor to define a height $\wh (\p)$    on  weighted projective spaces $\wP_{\w}^n (k)$.    We prove some of the basic properties of this weighted height, including an analogue of the Northcott's theorem for heights on projective spaces.  
\end{abstract}

\maketitle


\def\AA{\mathcal A}

\section{Introduction}

Most of the computations with genus 2 curves or genus 3 hyperelliptic curves,  whether occurring in number theory, mathematical physics, cryptography, or any other area, involve the corresponding tuple of invariants of binary forms.  An isomorphism class of such curves correspond to a projective point $\left[ J_{q_0} :  \dots : J_{q_n} \right]$ of modular invariants with degrees $q_0, \dots , q_n$ respectively. Of course this is true for all hyperelliptic or even superelliptic curves of any genus. In most of these computations picking the point $\left[ J_{q_0} :  \dots : J_{q_n} \right]$ with smallest coordinates is desirable; see for example computations in  \cite{MR3731039}, \cite{MR0114819},   \cite{c-m-sh} or all the algorithms in cryptography for genus $g=2$ and $g=3$ hyperelliptic curves. So how can we pick the  point with smallest coordinates or have some ordering on these moduli points in some reasonable way?  Since the ring of invariants of binary forms is a graded ring, the answer is equivalent to introducing some concept of the greatest common divisor  for weighted projective spaces similar to that of the $\gcd$ of a tuple of integers  in the projective space. If possible we would like to extend the analogy and 
introduce some concept of height in a weighted projective space similarly to the height in a regular projective space, which would make the ordering of points in a weighted projective space possible. The goal of this paper is to suggest a way to handle both of these questions.

In \cite{mandili-sh} was introduced the idea of the weighted common divisor on a tuple of integers with different weights, which was called the weighted greatest common divisor and denoted by $\wgcd$.    
%
In section 2 we give a precise definition of the concept of the weighted greatest common divisor and some of its properties. While the idea of the weighted greatest common divisor seems natural, surprisingly it has not appeared before in the literature.  Questions still remain on  efficient ways of computing such common divisor or whether such weighted greatest common divisor  has similar properties as the regular greatest common divisor in more general rings.  

For a a set of weights $\w=(q_0, \dots , q_n)$ and a number field $K$, the weighted greatest common divisor $\wgcd (\x)$ of 
a tuple $\x = (x_0, \dots , x_n)\in \O_K^{n+1}$ is  defined as the largest integer  $d\in \O_K$ such that $d^{q_i} $ divides $x_i$, for all $i=1, \dots , n$.   The absolute greatest common divisor $\awgcd (\x)$ is  defined as the largest real number   $d$ such that $d^{q_i} $ divides $x_i$, for all $i=1, \dots , n$.   
\cref{recombining} shows that this definitions are precise. 

In section 3 we apply this method to normalize points in weighted projective spaces. 
A point $\p =[x_0, \dots , x_n$ in the weighted projective space $\wP_\w^n (K)$ is said to be \textit{normalized}  if $\wgcd (x_0, \dots , x_n)=1$
and   \textit{absolutely normalized} if $\awgcd (x_0, \dots , x_n)=1$. It turns out that these \textit{normalizations} are unique up to multiplication by a root of unity (cf. \cref{unique}). Moreover, such normalization is unique for well-formed weighted projective spaces. Normalizing point in a weighted projective space this way gives a very efficient way of storing points in such spaces.  This idea was used in \cite{gen-2} and \cite{gen-3} to study the moduli space of genus 2 and genus 3 hyperelliptic curves. 

In section 4 we shift our attention to introducing heights in weighted projective spaces. 
The concept of height on a variety $\AA$ over a number field $K$ is a function $H: \AA (K) \to \R$ whose value at a point $P \in \AA(K)$  measures the arithmetic complexity of $P$. There are two properties that one would want in a height function:  i) there are only finitely many points of bounded height, ii) geometric properties are somewhat preserved. 

Heights on projective spaces are well known in the literature; see \cite{MR2162351}, \cite{MR3525576}, \cite{MR2216774} among many others.   For a point $P\in \P^n (\Q)$, we take integer projective coordinates  $P=[x_0: \dots: x_n]$ with $\gcd (x_0, \dots , x_n)=1$, then  the height is defined as 
\[ H(P) = \max \left\{   |x_0|, \dots , |x_n| \right\}. \]
The definition can be extended to any number field $K$ as follows
\[  H_K (P) = \prod_{v\in M_K} \max \left\{   |x_0|_v^{n_v}, \dots , |x_n|_v ^{n_v}\right\}. \]
where $M_K$ is the set of norms in $K$ and $n_v$ the local degree $[K_v:\Q_v]$. As an immediate consequence of the definition is the \textit{Northcott's theorem}, which says that there are only finitely many points $P \in \P^n (K)$, with height bounded by a constant $B$. A corollary of this statement is the \textit{Kronecker's theorem} which says that for any $\alpha \in K^\star$, $H_K (\alpha)=1$ if and only if $\alpha $ is a root of unity. In other words, there are only finitely many points of bounded height and bounded degree. 

Let $V_K$ be a projective subvariety of $\P^n (K)$ and $S \subset V_K$. In arithmetic,  height functions are used in two main ways:  i) To show that $S$ is finite, it is enough to show that it is a set of bounded height, ii) if $S$ is infinite, determine its density by estimating the growth of the \textit{counting function} 
$ N(S, B) = \# \{ P \in S \; : \; H_K (P) \leq B \}$.
The size of the set of points in $\P^n (K)$  is estimated by Schanuel's theorem. 

Weil extended the definition of height to  all projective varieties via ample divisors and provided an important connection between geometry and arithmetic. N\'eron and Tate introduced canonical heights for Abelian varieties. Perhaps one of the most popular uses of the machinery of heights is the proof of the Mordell-Weil theorem: \textit{For any Abelian variety $\AA_K$, the set of $K$-points $\AA (K)$ is a finitely generated Abelian group}. The main goal of this paper is to investigate how the machinery of heights for projective spaces can be extended to weighted projective spaces. Whether the weighted projective height introduced here can be interpreted in terms of blowups, along the lines of \cite{MR2162351}, will be the focus of future investigation.


Let $\w=(q_0, \dots , q_n)$ be a set of heights and $\wP^n(K)$ the weighted  projective space  over  a number field $K$ and $M_K$ the set of places of $K$.   Let  $\p \in \wP^n(K)$ a point such that  $\p=[x_0, \dots , x_n]$. We define the  \textbf{weighted height} of $\p$   as  
\[
\wh_K( \p ) := \prod_{v \in M_K} \max   \left\{   \frac{}{}   |x_0|_v^{\frac {n_v} {q_0}} , \dots, |x_n|_v^{\frac {n_v} {q_n}} \right\}
\]
The \textbf{weighted logarithmic height} of the point $\p$ is defined as follows
\[\wh^\prime_K(\p) := \log \wh_K(\p)=   \sum_{v \in M_K}   \max_{0 \leq j \leq n}\left\{\frac{n_v}{q_j} \cdot  \log  |x_j|_v \right\}.\]
We prove that $\wh_K( \p )$ is well defined and $\wh_K( \p )\geq 1$.  
Moreover,   if $K=\Q (\awgcd (\p))$ then   similarly to the projective space, 
\[
\wh_K (\p)=  \max_{0 \leq j \leq n}\left\{\frac{}{}|x_j|^{1/q_j}_\infty \right\}.
\]
when $\p$ is normalized and if  $L/K$ is a finite extension, then 
$
\wh_L(\p)=\wh_K(\p)^{[L:K]}.
$
The \textbf{absolute height} of a point $\p \in \wP^n (K)$ is defined as 
$\awh: \wP^n(\bar \Q)  \to [1, \infty)$ for    $\awh(\p) =\wh_K(\p)^{1/[K:\Q]}$.   It turns out that for weighted heights $\Q (\awgcd (\p))$ plays the role that the base field $\Q$ plays for regular projective height, see \cref{prop-5}.  This is no surprise since the greatest common divisor is in $\Q$ for projective heights. 
We are also able to consider the weighted  heights through the Weil height, via the map  $\phi: \wP_\w^n (K) \to \P^n (K)$, where 
\[ [x_0, \dots , x_n] \to \left[ x_0^{\frac q {q_0}}, \dots ,  x_n^{\frac q {q_n}}   \right]  \]
where $q=q_0\cdots q_n$. Then $\wh (\p) = H (\phi (\p)^{\frac 1 q}$, see \cref{wh-H}.
As in the projective space the weighted height is  invariant under Galois conjugation. In other words, for  $\p \in \wP^n(\overline \Q)$ and $\sigma \in G_{ \Q}$ we have $\wh (\p^\sigma) = \wh (\p)$ (cf. \cref{lem_galois_conj}). In \cref{thm_finite} we prove an analogue of Northcott's theorem for weighted heights. 

%
%


The weighted height seem to provide a powerful tool in studying the arithmetic properties of the weighted projective spaces. This could lead to many interesting results in many applications of such spaces. 


\medskip

\noindent \textbf{Aknowledgments:} We want to thank J. Silverman and J. Ellenberg for insightful comments and suggestions which significantly improved this paper.

\section{Weighted greatest common divisors}\label{sect-2}
Let $\x = (x_0, \dots x_n ) \in \Z^{n+1}$ be a tuple of integers, not all equal to zero.  Their greatest common divisor, denoted by $\gcd (x_0, \dots, x_n)$, is defined as the largest integer $d$ such that $d | x_i$, for all $i=0, \dots , n$.  

The concept of the weighted greatest common divisor of a tuple for the ring of integers $\Z$ was defined in \cite{mandili-sh}.   Let $q_0$, \dots, $q_n$ be positive integers.  A set of weights is called the ordered tuple 
$ \w=(q_0, \dots, q_n)$. 

Denote by  $r=\gcd (q_0, \dots , q_n)$ the greatest common divisor of $q_0, \dots , q_n$.  A \textit{weighted integer tuple} is a tuple $\x = (x_0, \dots, x_n ) \in \Z^{n+1}$ such that to each coordinate $x_i$ is assigned the weight $q_i$.  We multiply weighted tuples by scalars $\lambda \in \Q$ via 
\[ \lambda \star (x_0, \dots , x_n) = \left( \l^{q_0} x_0, \dots , \l^{q_n} x_n   \right) \]
For an ordered   tuple of integers  $\x=(x_0, \dots, x_n) \in \Z^{n+1}$, whose coordinates are not all zero, the \textbf{weighted greatest common divisor with respect to the set of weights} $\w$ is the largest integer $d$ such that 
\[ d^{q_i} \, \mid \, x_i, \; \; \text{for all }    i=0, \dots, n.\]
The first natural question arising from this definition  is to know if such integer $d$ does exist for any tuple $\x = (x_0, \dots, x_n ) \in \Z^{n+1}$. Clearly, it does exist because $x_i \leq d^{q_i}$ for all $i =0,\ldots,n$ and the largest integer is unique. We will denote by  $\wgcd (x_0, \dots, x_n) = \wgcd (\x)$.

Given  integer $a$ and non-zero integer $b$,  the integer part of the real number  $\frac {a}{b}$ is denote by 
$\left\lfloor \frac{a}{b}\right\rfloor$, that is, it is the unique integer satisfying:
\[ a = \left\lfloor \frac{a}{b} \right\rfloor b + r, \quad  0\leq r < b.\]
The next result  provides an algorithm to compute the weighted greatest common divisor.
 
\begin{prop}\label{existence_wgcd}  
For a weighted integer tuple $\x=(x_0, \dots, x_n)$ with weights $\w=(q_0, \dots, q_n)$ let  the factorization of the integers $x_i, \,( i=0,\ldots,n )$ into primes:
\[ 
x_i= \prod_{j=1}^t p_j^{\alpha_{j,i}}, \quad \alpha_{j,i} \geq 0, \,  \, j=1,\ldots, t
\] 
Then,  the weighted greatest common divisor $d = \wgcd (\x)$  is given by
\begin{equation}
d = \prod_{j=1}^t p_j^{\alpha_j}
\end{equation}
where,  
\begin{equation} 
\alpha_j = \min \left\{  \left\lfloor \frac{\alpha_{j,i}}{q_i} \right\rfloor, \, i =0, \ldots, n  \right\} \, \text{and} \,\,  j=1,\ldots, t.
\end{equation}
\end{prop}

\proof If  $d^{q_i} \mid x_i $,  then $d$ should be of the form $ \prod_{j=1}^t p_j^{\beta_j}$ for certain integers $\beta_j \geq 0$. On the other hand, for every prime $p_j$ and since 
$d^{q_i} \mid x_i $, then  
\[ \beta_j \leq \frac{\alpha_{j,i}}{q_i}, \, i=0, \ldots, n. \]
Now, the proof  is straightforward. 
\endproof
%

In the next we illustrate the method by a toy example:
\begin{exa}
Consider the set of weights $\w=(3, 2)$ and the tuple 
\[ \x = \left(   1440, 700 \right) = \left (2^5 \cdot 3^2  \cdot 5 \cdot 7^0,\, 2^2 \cdot 3^0  \cdot 5^2 \cdot 7 \right) \in \Z^2.\]
Then,  $\wgcd (\x) = d = 2^{\alpha_1} \cdot 3^{\alpha_2}  \cdot 5^{\alpha_3} \cdot 7^{\alpha_4} $, where 
\[
\begin{split}
 \alpha_1 &= \min \left\{  \left\lfloor \frac{5}{3} \right\rfloor,  \left\lfloor \frac{2}{2} \right\rfloor \right\} = 1, \quad  \alpha_2 = \min \left\{ \left\lfloor \frac{2}{3} \right\rfloor,  \left\lfloor \frac{0}{2} \right\rfloor \right\} = 0, \\ 
 \alpha_3 & = \min \left\{  \left\lfloor \frac{1}{3} \right\rfloor,  \left\lfloor \frac{0}{2}\right\rfloor \right\} = 0, \quad     \alpha_4 = \min \left\{  \left\lfloor \frac{0}{3}\right\rfloor,  \left\lfloor \frac{1}{2}\right\rfloor \right\} = 0. 
\end{split}
\] 
 Then $d = 2$.
 \qed
\end{exa} 

An integer tuple $\x=(x_0, \ldots , x_n) \in \Z^{n+1} $ with $\wgcd (\x)=1$ is called \textbf{normalized}. 
For an integer  tuple $\x=(x_0, \dots, x_n)$ exist integers  $(y_0, \dots, y_n) \in \Z^{n+1}$ such that
\[ \gcd (x_0, \dots x_n ) = x_0 y_0 + \dots + x_n y_n. \]
For   weights $\w=(q_0, \dots, q_n)$, we have that  $\wgcd (\x) | \gcd(\x)$, say 
\[ \gcd (x) = \lambda \cdot \wgcd (\x).\]
Then,
\[ 
\wgcd (\x) = \left( \frac {x_0} \lambda   \right) y_0  + \left( \frac {x_1} \lambda   \right) y_1  + \cdots + \left( \frac {x_n} \lambda   \right) y_n  =
\sum_{i=0}^n \left( \frac {x_i} \lambda \right) \,  y_i \]
Notice that each $\frac {x_i} \lambda$ is an integer from the definition of the $\wgcd (\x)$. 


The \textbf{absolute weighted greatest common divisor} of  an integer tuple $\x=(x_0, \dots , x_n)$ with respect to the set of weights $\w=(q_0, \dots , q_n)$ is the largest 
real number $d$ such that %
\[ d^{q_i} \in \Z  \quad {\text and} \quad   d^{q_i}\, \mid \, x_i, \; \; \text{for all }    i=0, \dots n.\]
Again, the natural question arising from this  new definition  is to know if such real number $d$  does exist for any tuple $\x = (x_0, \dots, x_n ) \in \Z^{n+1}$. 
Since  $x_i \leq d^{q_i}$ and   there are a finite number of divisors of $x_i$, for all $i =0,\ldots,n$, so we are looking for the largest real number of finite set of numbers and,  the largest is unique.   We will denote by the absolute weighted greatest common divisor by $\awgcd (x_0, \dots, x_n)$.

In order to provide a method to compute the $\awgcd (x_0, \dots, x_n)$, we  need the  following technical  elementary result.

\begin{lem}\label{cecilia}  
Let $d\in \R^+$ a positive real number. If there exists a positive integer $m$ such that  $d^m$ is a positive integer, then $d = z^{1/m}$  for some positive integer $z$. 
Moreover if $m$ is the smallest integer such that $d^m$ is a positive integer,  then  any positive integer $q$ verifying  $d^q$ is  a positive integer,  is a multiple of $m$.
\end{lem} 

The next  result   provides a method to compute  the absolute weighted greatest common divisor:
 
\begin{prop}\label{existence_awgcd}  
For a weighted integer tuple $\x=(x_0, \dots, x_n)$ with weights $\w=(q_0, \dots, q_n)$ let  the factorization of the integers $x_i, \,( i=0,\ldots,n )$ into primes:
\[x_i= \prod_{j=1}^t p_j^{\alpha_{j,i}}, \quad \alpha_{j,i} \geq 0, \,  \, j=1,\ldots ,t \]
Then,  the  absolute weighted greatest common divisor $d = \awgcd (\x)$  is given by
\[
d = \left (\prod_{j=1}^t p_j^{\alpha_j} \right)^{ \frac{1}{q}}
\]
where,  $q=\gcd(q_0,\ldots,q_n)$, \, $q_i = q \cdot \bar q_i$ and 
\[ \alpha_j = \min \left\{ \left\lfloor \frac{\alpha_{j,i}}{\bar q_i}\right\rfloor, \, i =0,\ldots, n \right\} \, \text{and} \,\,  j=1,\ldots, t.\]
\end{prop}

\proof    From  \cref{cecilia} we have that  $d^{q}\, \mid \, x_i, \, i=0,\ldots, n$. Then  $d$ should be of the form
 $d= \left( \prod_{j=1}^t p_j^{\beta_j}\right )^\frac{1}{q}$ 
 for certain integers $\beta_j \geq 0$. On the other hand, for every prime $p_j$ and since 
$d^{q_i} \mid x_i $, then  
\[ \beta_j \leq \frac{\alpha_{j,i}}{\bar q_i}, \, i=0,\ldots,n.\]
%
%
Again, the rest of the proof is immediate.

\endproof

\begin{exa}
Consider the set of weights $\w=(6, 8)$ and  the tuple 
\[ \x = \left(   2^{15} \cdot 5^{12} ,  2^{26} \cdot 5^{13} \right) \in \Z^2.\]
Then $q = \gcd(6,8) =2, \, p_1 = 2, \,  p_2 = 5, \, t=2$   and $\bar q_1= 3, \bar q_2 =4$. Then,  $\awgcd (\x) = d = \left( 2^{\alpha_1} \cdot 5^{\alpha_2} \right)^{\frac{1}{2}} $, where 
\[ \alpha_1 = \min \left\{  \left\lfloor \frac{15}{3} \right\rfloor,  \left\lfloor \frac{26}{4} \right\rfloor \right\} = 5, \qquad  \alpha_2 = \min \left\{  \left\lfloor \frac{12}{3}\right\rfloor,  \left\lfloor \frac{13}{4}\right\rfloor \right\} = 3.\]
Hence $d= 2^{\frac{5}{2}} \cdot 5^{\frac{3}{2}}= \sqrt{2^5 \cdot 5^3}$. On the other hand,  $\wgcd (\x) = 2^2 \cdot 5$. As expected,  $\wgcd (\x) \leq \awgcd (\x)$. 
\end{exa}

The next example comes from the theory of invariants of binary sextics.
 
\begin{exa}
Consider the set of weights $\w=(2, 4, 6, 10)$ and a tuple 
\[ \x = \left(   3 \cdot 5^2 , 3^2 \cdot 5^4, 3^3 \cdot 5^6, 3^5 \cdot 5^{10} \right) \in \Z^4.\]
Then,  $\wgcd (\x) = 5$ and \; $\awgcd (\x) = 5 \cdot \sqrt{3}$.  
\end{exa} 
An integer  tuple $\x$ with $\awgcd (\x)=1$ is called \textbf{absolutely normalized}. We summarize in the following lemma.
\begin{lem}
For any weighted integral tuple $\x= (x_0, \dots , x_n) \in \Z^{n+1}$ such that $\w(x_i) = q_i$, $i=0, \dots, n$,  the tuple
$  \y = \frac 1 {\wgcd (\x) } \star \x, $
is integral and normalized. Moreover, the tuple 
$   \bar \y = \frac 1 {\awgcd (\x) } \star \x, $
is also integral and absolutely normalized. 
\end{lem}

Normalized tuples are unique up to a multiplication of $q$-root of unity (cf. \cref{unique}), where $q=\gcd (q_0, \dots , q_n)$. 
It is worth noting that a normalized tuple is a tuple with "smallest" integer coordinates (up to multiplication by a unit).  
We will explore this idea of the "smallest coordinates" in the coming sections. 

There are a few natural questions that arise with the weighted greatest common divisor of a tuple of integers.  We briefly discuss the two main ones:  \\

\noindent \textbf{Problem  1:}  The greatest common divisor can be computed in polynomial time using the Euclidean algorithm.  Determine the fastest way to compute the weighted greatest common divisor and the absolute weighted greatest common divisor. \\

\noindent \textbf{Problem  2:}  The greatest common divisor is uniquely determined for unique factorization domains.  Define the concept of the weighted greatest common divisor in terms of ring theory and determine the largest class of rings where it is uniquely defined (up to multiplication by a unit). \\

\subsection{Complexity of computing the  weighted greatest common divisor}
Let  $\x=(x_0, \dots, $ $\dots , x_n) \in \Z^{n+1}$ and weights $\w=(q_0, \dots, q_n)$. Then \cref{existence_wgcd}    and \cref{existence_awgcd} provide a method to compute $\wgcd (\x)$ and $\awgcd (\x)$ (respectively)  for weights $\w=(q_0, \dots, q_n)$. In both, we have to compute  
the integer factorization into primes of all elements of the tuple $\x$. Of course, this is not very efficient comparing with the computation of $\gcd (\x))$. On the other hand, there are several indications that we can not avoid  factoring. For instance, we have that 
$\wgcd(0, \dots,0, x_n)$ is  $\wgcd(x_n)$, then we are looking for the largest factor $d$ of $x_n$ such that 
$d^{q_n}$ divides $x_n$.

Alternatively, we can factor only an integer, instead of $n+1$, and then recombining factors in an appropriate and clever way gives us the following. 

\begin{lem}\label{recombining} 
With the above notation, let $g=\gcd (x_0, \dots, x_n) $ and  $g= \prod_{i=1}^r p_i^{s_i}$ its prime factorization. 
\begin{enumerate}
\item For $i =1, \ldots, r$, let 
\[ 
\beta_i=  \min   \left\{ \left\lfloor \frac{s_i}{q_j} \right\rfloor \,:\, j =0, \ldots, n \right\}.
\]
Then,  the weighted greatest common divisor $d = \wgcd (\x)$  is given by
\[
d = \prod_{i=1}^r p_i^{\alpha_i},
\]
where  $\alpha_i$ are the largest integers such that $d^{q_i}$ divides $x_i$ and $\alpha_i \leq \beta_i$. 
\item  Let $q = \gcd(q_0,\ldots, q_n), \, q_j = q \cdot \bar q_j, \, j =0,\ldots, n$ and for  $i=1,\ldots,r$ let
\[
\beta_i = \min \left\{  \left\lfloor \frac{s_i} {\bar q_j} \right\rfloor, \,  j=0,\ldots, n \right\}
\]
Then,  the absolute weighted greatest common divisor $d= \awgcd (\x)$   is 
%
\[ 
d = \left(\prod_{i=1}^r p_i^{\alpha_i}\right)^\frac{1}{q}
 \]
where  $\alpha_i$ are the largest integers such that $d^{q_i}$ divides $x_i$ and $\alpha_i \leq \beta_i$. 
\end{enumerate}

\end{lem}

\begin{proof}   To prove 1)  we have that  $d^{q_i}$ divides $x_i$, then $d$ divides $x_i$ and it implies $d$ divides $g$. Now, the  proof is straightforward.
To prove 2) we have that $d^{q_i}$ divides $x_i$ and from \cref{cecilia},  $d^q$ divides $x_i$ and it implies that $d^q$ divides $g$. The rest proof is immediate.
\end{proof}

It is  well known that the number of divisors  $D(m)$ of integer  $m$ is  $m^{{\small o}(1)}$. So, in the worst case the previous result  \cref{recombining}  get an exponential  time complexity.

\subsection{Weighted greatest common divisor over general rings}

Let $R$ be a commutative ring with identity.  Consider the set of weights $\w= (q_0, \dots , q_n)$ as in the previous section  and a tuple $\x \in R^{n+1}$. For any $\alpha \in R$, the ideal generated by $\alpha$ is denoted by $(\alpha)$.  The \textbf{weighted greatest common divisor ideal} is defined as 
\[ \J (\x) = \bigcap_{ \left( \p^{q_i} \right)   \supset (x_i)}  \p \]
over all primes $\p$ in  $R$.  If $R$ is a PID then the $\wgcd (\x)$ is the generator of the principal ideal $\J (\x)$. 
In general, for $R$ a unique factorization domain, for any point $x = (x_1, \dots, x_n) \in R^n$ we let $r=\gcd (x_0, \dots , x_n)$.  
 Factor $r$ as a product of primes, say 
$ r = u \cdot \prod_{i=1}^s \p_i,$ 
where $u$ is a unit and $\p_1, \dots \p_s$ are primes. Then the weighted gcd $\wgcd(\x)$ is defined as
\[ 
\wgcd(\x) = \prod_{ \stackrel{i=1}{p^{q_i} | x_i} }^s \p
\]
Thus, the weighted gcd (as the common gcd) is defined up to multiplication by a unit. 
The \textbf{absolute weighted greatest common divisor ideal} is defined as 
\[ \bar \J (\x) = \bigcap_{ \left( \p^{\frac {q_i} r} \right) \supset (x_i)}  \p \]
over all primes $\p$ in  $R$.  

The above definitions can be generalized to GCD domains.  An  integral domain  $R$ is called  a \textbf{GCD domain}  if  any two elements of $R$ have a greatest common divisor.  Examples of GCD-domains include unique factorization domains and valuation domains, see \cite{kaplansky} for more details.    
%


\def\b{\beta}

\subsection{Generalized weighted greatest common divisors}
Following on the ideas of  \cite{MR2162351}  we give a brief review of the generalized greatest common divisors and how they can be defined for weighted greatest common divisors as well. Let $k$ be a number field, $\O_k$ its ring of integers, $M_k$ the set of absolute values of $k$, $M_k^0$ all non-archimedian places, and $M_k^\infty$ archimedian places of $M_k$. 

For any two elements $\a, \b \in \O_k$ the greatest common divisor is defined as 
\begin{equation} 
\gcd (\a, \b) = \prod_{p \in \O_k} p^{\min \{ \nu_p (\a), \, \nu_p (\b) \}} ,
\end{equation}
where $\nu_p$ is the valuation corresponding to the prime $p$; see \cite{MR2162351} for details.   The logarithmic $\gcd$ is 
\begin{equation} 
\log \gcd (\a, \b) = \sum_{\nu \in M_k^0} \min \, \{ \nu (\a), v(\b) \} 
\end{equation}
For a valuation $\nu \in M_k$, define 
\begin{equation}
\begin{split}
\nu^+ : k & \longrightarrow [0, \infty], \\
\alpha & \longrightarrow  \max \{ v(\alpha), 0\}. \\
\end{split}
\end{equation}
The \textbf{generalized logarithmic greatest common divisor} of two elements $\a, \b \in k$ is defined as 
\begin{equation} 
h_{gcd} (\a, \b) = \sum_{\nu \in M_k} \min \{ \nu^{+} (\a), \nu^{+} (\b) \}. 
\end{equation}
Notice that $\nu^+$ can be viewed as a height function on $\P^1 (k) = k \cup \{ \infty \}$,  where we set $\nu^+ (\infty) = 0$. This leads to the generalized logarithmic greatest common divisor being viewed also as a height function:
\begin{equation}
\begin{split}
G_\nu : \P^1 \times \P^1 &   \to [0, \infty] \\
(\a, \b) & \to \min \{ \nu^+ (\a), \nu^+ (\b) \} \\
\end{split}
\end{equation}
In view of the above we have
\[ h_{gcd} (\a, \b) = \sum_{\nu \in M_k} G_\nu. \]
In \cite{MR2162351} it was given a theoretical interpretation of the function $G_\nu$ in terms of blowups. 
%

\begin{lem}
For a weighted integer tuple $\x=(x_0, \dots, x_n) \in \O_k^{n+1}$    the weighted greatest common divisor  is given by
\[
\wgcd_\w (\x) = \prod_{p \in \O_k}    p^{ \min \left\{ \left\lfloor  \frac {\nu_p (x_0)}  {q_0}   \right\rfloor,  \dots ,  \left\lfloor  \frac {\nu_p (x_n)}  {q_n}   \right\rfloor  \right\}     }
\]
\end{lem}

\proof  The proof is elementary.   From \cref{existence_wgcd} we have that 
\[ \wgcd_\w (\x) = \prod_{j=1}^t p_j^{\a_j},
\]
where $\a_{i, j} = \min \left\{  \left\lfloor     \frac {\a_{j, i}} {q_i} \right\rfloor \; | \: i=0, \ldots , n    \right\}$, for each $j=1, \ldots , t$.  But $\a_{j, i} = \nu_p (x_i)$, for each $i=0, \ldots , n$.  The rest follows. 

\qed

As above, the logarithmic weighted greatest common divisor is 
\[ 
\log \wgcd_\w (\x) = \sum_{\nu \in M_k^0}  \min \left\{ \left\lfloor  \frac {\nu_p (x_0)}  {q_0}   \right\rfloor,  \dots ,  \left\lfloor  \frac {\nu_p (x_n)}  {q_n}   \right\rfloor  \right\} 
\]
Consider now $\x=(x_0, \dots, x_n) \in k^{n+1}$ with weights $\w=(q_0, \dots, q_n)$.  The \textbf{generalized weighted greatest common divisor} is defined as follows 
\[
 \hwgcd (\x) = \prod_{p \in \O_k}    p^{ \min \left\{ \left\lfloor  \frac {\nu_p^+ (x_0)}  {q_0}   \right\rfloor,  \dots ,  \left\lfloor  \frac {\nu_p^+ (x_n)}  {q_n}   \right\rfloor  \right\}     }
\]
and the   logarithmic weighted greatest common divisor is 
\[ 
\log \hwgcd (\x) = \sum_{\nu \in M_k^0}  \min \left\{ \left\lfloor  \frac {\nu_p^+ (x_0)}  {q_0}   \right\rfloor,  \dots ,  \left\lfloor  \frac {\nu_p^+ (x_n)}  {q_n}   \right\rfloor  \right\} 
\]
Let  $q= \gcd (q_0, \dots , q_n)$ and $\bar q_i = \frac {q_i} q$. Hence, we get a new set of well-formed weights $\bar q= \left(  \bar q_0, \bar q_1, \dots , \bar q_n \right)$.

The factorization of coordinates of  $\x$   into primes is $x_i = \prod_{p\in \O_k} p^{\nu_p (x_i)}$, for $i=0, \ldots , n$.  Then we have:

\begin{lem}
The absolute weighted greatest common divisor  is
\[
\awgcd (\x) =  \left( \prod_{p\in \O_k} p^{ \min \, \left\{ \left\lfloor \frac {\nu_p(x_0)} {\bar q_0}   \right\rfloor,    \dots ,  \left\lfloor \frac {\nu_p(x_n)} {\bar q_n}   \right\rfloor  \right\}     } \right)^{\frac 1 q}
\]
\end{lem}

\proof
The proof is similar to the previous Lemma, but using \cref{existence_awgcd}. 

\qed

Accordingly we define the  \textbf{generalized absolute weighted greatest common divisor}  by
\[ 
\hawgcd (\x) = \frac 1 q \, \prod_{p\in \O_k} p^{ \min \, \left\{ \left\lfloor \frac {\nu_p^+ (x_0)} {\bar q_0}   \right\rfloor,    \dots ,  \left\lfloor \frac {\nu_p^+ (x_n)} {\bar q_n}   \right\rfloor  \right\}     }
\]
and the   logarithmic absolute weighted greatest common divisor is 
\[ 
\log \hwgcd (\x) = \sum_{\nu \in M_k^0}  \min \left\{ \left\lfloor  \frac {\nu_p^+ (x_0)}  {\bar q_0}   \right\rfloor,  \dots ,  \left\lfloor  \frac {\nu_p^+ (x_n)}  {\bar q_n}   \right\rfloor  \right\} 
\]
Let us see an example.
\begin{exa} Let $\w=(2,4,6,10)$ and $\p \in \wP_\w^3 (\Q)$ such that 
\[ 
\p = [ 2^3\cdot 3^2\cdot 7^3 ;  2^5\cdot 3^7\cdot 7 ;  2^7\cdot 3^7\cdot 7^3 ; 2^{11}\cdot 3^{13}\cdot 7^5].
\]
Then  $\wgcd (\p) = 2 \cdot 3$  and   $\awgcd (\p) = 2\cdot 3$. 
The normalized point is
\[ \bar \p = [ 2 \cdot 7^3 ; 2 \cdot 3^3 \cdot 7 ;  2\cdot 3 \cdot 7^3 ; 2 \cdot 3^3 \cdot 7^5]
\]
Using the previous two lemmas we have 
\[
\begin{split}
 \wgcd (\p) & = 2^{ \a_1 } \cdot 
3^{ \a_2    } 
\cdot 7^{ \a_3 }  = 2\cdot 3 \cdot 7^0= 2\cdot 3. 
\end{split}
\]
where 
\[
\a_1=\min \left\{ \left\lfloor  \frac 3 2  \right\rfloor, \left\lfloor  \frac 5 4  \right\rfloor, \left\lfloor  \frac 7 6  \right\rfloor  , \left\lfloor  \frac {11} {10}  \right\rfloor \right\}, \;
\a_2=\min \left\{ \left\lfloor  \frac 2 2  \right\rfloor, \left\lfloor  \frac 7 4  \right\rfloor, \left\lfloor  \frac 7 6  \right\rfloor  , \left\lfloor  \frac {13} {10}  \right\rfloor \right\},
\]
\[
\a_3=\min \left\{ \left\lfloor  \frac 3 2  \right\rfloor, \left\lfloor  \frac 1 4  \right\rfloor, \left\lfloor  \frac 3 6  \right\rfloor  , \left\lfloor  \frac {5} {10}  \right\rfloor \right\}  
\]
Similarly for $\awgcd$ we have
\[
\begin{split}
 \awgcd (\p) & = \left( 2^{ \b_1 } \cdot  3^{ \b_2  }  \cdot 7^{ \b_3  } \right)^{\frac 1 2}  = \left(    2^2 \cdot 3^2 \cdot 7^0 \right)^{\frac 1 2 } = 2\cdot 3, 
\end{split}
\]
where
\[
\b_1= \min \left\{ \left\lfloor  \frac 3 1  \right\rfloor, \left\lfloor  \frac 5 2  \right\rfloor, \left\lfloor  \frac 7 3  \right\rfloor  , \left\lfloor  \frac {11} {5}  \right\rfloor \right\},  \;
\b_2=\min \left\{ \left\lfloor  \frac 2 1  \right\rfloor, \left\lfloor  \frac 7 2  \right\rfloor, \left\lfloor  \frac 7 3  \right\rfloor  , \left\lfloor  \frac {13} {5}  \right\rfloor \right\}
\]
\[ 
\b_3=\min \left\{ \left\lfloor  \frac 3 1  \right\rfloor, \left\lfloor  \frac 1 2  \right\rfloor, \left\lfloor  \frac 3 3  \right\rfloor  , \left\lfloor  \frac {5} {5}  \right\rfloor \right\} 
\]
Of course this is no surprise sine the normalization $\p$ can be seen easily that is absolutely normalized. 
\qed
\end{exa}

\section{Normalized points in weighted projective spaces}
Let $K$ be a field   and  $(q_0, \dots , q_n) \in \Z^{n+1}$ a fixed tuple of positive integers called \textbf{weights}.   Consider the action of $K^\star = K \setminus \{0\}$ on $\A^{n+1} (K)$ as follows
\begin{equation}\label{equivalence}
 \lambda \star (x_0, \dots , x_n) = \left( \l^{q_0} x_0, \dots , \l^{q_n} x_n   \right) 
\end{equation}
for $\l\in K^\ast$.  The quotient of this action is called a \textbf{weighted projective space} and denoted by   $\wP^n_{(q_0, \dots , q_n)} (K)$. 
The space $\wP_{(1, \dots , 1)} (K)$ is the usual projective space.  The space $\wP_w^n$ is called \textbf{well-formed} if   
\[ \gcd (q_0, \dots , \hat q_i, \dots , q_n)   = 1, \quad \text{for each } \;  i=0, \dots , n. \]
While most of the papers on weighted projective spaces are on well-formed spaces, we do not assume that here.   We will denote a point $\p \in \wP_w^n (K)$ by $\p = [ x_0 : x_1 : \dots : x_n]$. 

Weighted projective spaces are   interesting since we can present a non singular algebraic variety as a hypersurface in a weighted projective space and deal with it as it would be a nonsingular hypersurface in a weighted projective space. For more on weighted projective spaces one can check \cite{MR879909}, \cite{MR627828}, \cite{MR2852925}, \cite{igor} among many others.


In projective spaces, by means of the Veronese embedding, we could embed the same variety in different projective spaces. It turns out that we can do the same for varieties embedded in weighted projective spaces.

As above we let $k$ be a field.  Let $R = \oplus_{ i \geq 0} R_i$ be a graded ring.  We further assume that 
\begin{itemize}
\item[(i)] $R_0=k$ is the ground field

\item[(ii)]  $R$ is finitely generated as a ring over $k$

\item[(iii)]  $R$ is an integral domain
\end{itemize}

Consider the polynomial ring $k[x_0, \dots , x_n]$ where each $x_i$ has weight $\wt x_i = q_i$.  Every polynomial is a sum of monomials $x^m= \prod x_i^{m_i}$ with weight $\wt (x^m) = \sum m_i q_i$.  A polynomial $f$ is \textbf{weighted homogenous of weight $m$} if every monomial of $f$ has weight $m$.  

An ideal in a graded ring  $I \subset R$ is called \textbf{graded} or \textbf{weighted homogenous} if $I = \oplus_{n\geq 0} I_n$, where $I_n = I\cap R_n$.  Hence, $R = k[x_0, \dots , x_n]/I$, where $\deg x_i = q_i$ and $I$ is a homogenous prime ideal. 

To the prime ideal $I$ corresponds an irreducible affine variety $CX= \Spec R = V_a (I) \subset \A^{n+1}$.

\begin{defi}\label{w.h.p}
A polynomial $f(x_0, \dots , x_n)$ is called \textbf{weighted homogenous} of degree $d$ if  it satisfies the following
\[  f(\l^{q_0} x_0, \l^{q_1} x_1, \dots , \l^{q_n} x_n) = \l^d f(x_0, \dots , x_n). \]
\end{defi}
\noindent Let us consider a simple example of weighted homogenous polynomials. 
\begin{exa} Let us consider a binary weighted form with weighted degree $d$ and let $w = (q_0, q_1)$ be  respectively the weights of $x_0$ and $x_1$. Then 
\[f(x_0, x_1) = \sum_{d_0, d_1} a_{d_0, d_1}  x_0^{d_0} x_1^{d_1},  \, \, \text{ such that }   \, \, d_0 q_0+ d_1 q_1 = d \]
and in decreasing powers of $x_0$ we have
\[ f(x_0, x_1)  =  a_{d/q_0, 0}  x_0^{d/q_0} + \dots + a_{d_0, d_1} x_0^{d_0} x_1^{d_1} + \dots+  a_{0, d/q_1} x_1^{d/q_1}\] 
By dividing this polynomial with $x_1^{d/q_1}$ and making a change of coordinates $X = x_0^{q_1} / x_1^{q_0}$ we get 
\begin{equation}
\begin{split}f(x_0, x_1)  &= a_{d/q_0, 0}  x_0^{d/q_0} + \dots + a_{d_0, d_1} x_0^{d_0} x_1^{d_1} + \dots+  a_{0, d/q_1} x_1^{d/q_1}\\
& =  a_{d/q_0, 0}  \frac{x_0^{d/q_0} }{x_1^{d/q_1}} + \dots +  a_{d_0, d_1} \frac{ x_0^{d_0} x_1^{d_1} }{x_1^{d/q_1}} + \dots+  a_{0, d/q_1}\\
& =  a_{d/q_0, 0} X^{d/q_0q_1} + \dots +  a_{d_0, d_1}  X^{d_0/q_1}+ \dots+  a_{0, d/q_1} = f(X)
\end{split}
\end{equation}
\end{exa}
Notice that the condition $f(P)=0$ is defined on the equivalence classes of Eq.~\eqref{equivalence}. We define the quotient 
$V_a (I)\setminus\{0\}$ by the above equivalence by $V_h (I)$, where $h$ stands for homogenous.  Then, we denote 
$X= \Proj R = V_h (I) \subset \wP_{\w}^n(k)$.  It is a projective variety. Notice that $CX$ above is the \textbf{affine cone} over the projective variety $V_h (I)$. 

Next we will define truncated rings and see the role that they play in the Veronese embedding. 
Define the $d$'th truncated ring $R^{[d]} \subset R$ by
\[ R^{[d]} = \bigoplus_{d |n} R_n  = \bigoplus_{i \geq 0} R_{di}, \]
Hence, $R^{[d]}$ is a graded ring and the elements have degree $di$ in $R$ and  degree $i$ in $R^{[d]}$.   If $R$ is a graded ring then its subring $R^{[d]}$ is called the $d$-th Veronese subring.

\begin{exa}
Let $R = k [x, y]$ with $wt(x) = wt (y) = 1$. Then,
\[ R^{[2]} = \bigoplus_{i \geq 0} R_{2i} = \bigoplus_{i \geq 0} \left \{f(x, y) \in k[x, y]  \left | \frac{}{} \right. \mbox{deg }(f) = 2i \right \}. \]
Notice that the even degree polynomials in $k[x, y]$ are generated by $x^2$, $xy$, and $y^2$ hence we have that 
\[ R^{[2]} = k[x^2, xy, y^2] \cong k [ u, v, w ] \big / \< uw -v^2\rangle \]
Now, if we consider the projective spaces we have that 
\[ \Proj \, \,( k[x, y] )= \P_{(1, 1)} =  \P^1\]
while
\[ \Proj \, \, (k [ u, v, w ] \big / \< uw -v^2\rangle  )=  V( uw -v^2 ) \subseteq \P_{(1, 1, 1)} = \P^2 \]
Hence we have that,
\[\P^1 (k) =  \Proj ( k[x, y] ) \cong \Proj \, \,( k[x, y]^2 ) \subseteq \P^2 (k). \]
This is exactly the degree-2 Veronese embedding of \, $\P^1(k) \hookrightarrow \P^2 (k)$.  The truncation of graded rings in this case corresponds to the degree-$2$ Veronese embedding. 
\end{exa}

The proof of the following lemma can be found in \cite{igor}. 

\begin{lem} \label{proj_iso}
Let $R$ be a graded ring and $d \in \N$. Then,
\[ \Proj   R \cong \Proj   R^{[d]}  \]
\end{lem} 
 
For some large enough $N$ and using the above \cref{proj_iso} we can embed a weighted projective space $\wP_w$ into a \lq \lq straight " projective space $\P^N$.  

\begin{prop} 
Consider the weighted polynomial ring $R = k [x_0, \dots, x_n]$ , where $q_0, \dots, q_n$ are positive integers such that the weight of  $x_i$   is $q_i$ and  $d= \gcd (q_0, \dots , q_n)$.  The following are true: 

i)  $R^{[d]}  = R$.  Thus, 
\[ \wP^n_{(q_0, \dots, q_n) }(R)  = \wP^n _{ \left(\frac{q_0}{d}, \dots, \frac{q_n}{d} \right)} (R).\]  

ii) Suppose that $q_0, \dots, q_n$ have no common factor, and that $d$ is a common factor of all $a_i$ for $i \neq j$ (and therefore coprime to $a_j$). Then the $d$'th truncation of $R$ is the polynomial ring 
\[ R^{[d]} = k [x_0, \dots, x_{j-1}, x_j^d, x_{j+1}, \dots, x_n].\]
Thus, in this case 
\[ \wP^n_{(q_0, \dots, q_n) } (R)=   \wP^n _{\left(\frac{q_0}{d}, \dots, \frac{q_{j-1}}{d}, q_j, \frac{q_{j+1}}{d}, \dots, \frac{q_n}{d}\right)}(R^{[d]}).\]

In particular by passing to a truncation $R^{[d]}$ of $R$ which is a polynomial ring generated by pure powers of $x_i$, we can always write any weighted projective space as a well formed weighted projective space. 
\end{prop}

 \begin{proof}
  i) If $d | q_i$ for all $i = 0, \dots, n$ then the degree of every monomial is divisible by $d$ and so part i) is obvious. Hence, the truncation does not change anything. 
 
 ii)  Since $d | q_i$ for every $i \neq j$ then $x_i \in \R^{[d]}$ for every $i \neq j$. But the only way that $x_j$ can occur in a monomial with degree divisible by $d$ is as a $d$'th power. Given
\[R = k [x_0, \dots, x_j, \dots,  x_n]\]
then 
\[R^{[d]} = k [x_0, \dots, x_j^d, \dots,  x_n]\]
and  
 \[\begin{split} 
 \wP^n_{(q_0, \dots, q_n) }(R) &= \Proj \, \,  k_w [x_0, \dots, x_j, \dots,  x_n] \iso\Proj \, \,  k_{w/d} [x_0, \dots, x_j^d, \dots,  x_n] \\
 & = \wP^n _{\left(\frac{q_0}{d}, \dots, \frac{q_{j-1}}{d}, q_j, \frac{q_{j+1}}{d}, \dots, \frac{q_n}{d}\right)}(R^{[d]}).
 \end{split}\]
 This completes the proof. 
 \end{proof}
 
Hence, the  above result  shows that any weighted projective space is isomorphic to a well formed weighted projective space.

\def\q{\mathfrak q}
\def\size{\mathfrak s}

For the rest of this paper we will always assume that $R$ is the ring of integers $\O_k$ for some number field $k$. 
We will call a point $\p \in \O_k^n$ a \textbf{normalized point} if the weighted greatest common divisor of its coordinates is 1.  Similarly an \textbf{absolutely normalized point} is called a point $\p$ such that $\awgcd (\p) =1$.

\begin{lem}\label{unique} 
Let $\w=(q_0, \dots , q_n)$ be a set of weights and $q=\gcd (q_0, \dots , q_n)$.   For any point $\p \in \wP_w^n (k)$, there exists its normalization given by 
\[ \q = \frac 1 {\wgcd (\p)} \star \p .
\]
Moreover, this normalization  is unique up to a multiplication by a $q$-root of unity.
\end{lem}

\proof  Let $\p = [x_0 : \dots , x_n ]  \in \wP_w^n (k)$ and $\p_1 = [\alpha_0 : \dots : \alpha_n ]$ and $\p_2 = [ \beta_0 : \dots : \beta_n]$ two different normalizations of $\p$. Then exists non-zero $\l_1, \l_2 \in k$   such that 
\[ 
\p = \l_1 \star \p_1 = \l_2 \star  \p_2,
\]
or in other words 
\[ 
(x_0, \dots , x_n) = \left( \l_1^{q_0} \alpha_0 , \dots , \l_1^{q_i} \alpha_i , \dots \right) =  \left( \l_2^{q_0} \beta_0 , \dots , \l_2^{q_i} \beta_i , \dots \right).
\]
Thus,
\[ 
  \left(   \alpha_0 , \dots ,   \alpha_i , \dots , \alpha_n \right) =  \left( r^{q_0} \beta_0 , \dots , r^{q_i} \beta_i , \dots , r^{q_n} \beta_n \right).
\]
for $r= \frac {\l_2} {\l_1} \in k$. Thus, $r^{q_i}=1$ for all $i=0, \dots , n$.  Therefore,  $r^q=1$.  This completes the proof.   
\qed

Thus we have the following:

\begin{cor}\label{cor-1}
Points  in a well-formed weighted projective space  $ \wP_w^n (k)$ have unique  normalizations. 
\end{cor}

Here is an example which illustrates \cref{unique}.

\begin{exa}
Let $\p = [x_0, x_1, x_2, x_3]  \in \wP_{(2,4,6,10)}^3 (\Q)$ be a normalized point.  
Hence, 
\[ \wgcd (x_0, x_1, x_2, x_3 ) =1.\]
Since $q = \gcd (2, 4, 6, 10)=2$, then we can take $r$ such that   $r^2=1$.  Hence, $r=\pm 1$.   Therefore,  the point
\[ (-1)  \star \p = [-x_0 : x_1 : -x_2 : -x_3 ] \]
is also normalized.  

However, if $\p = [x_0, x_1, x_2, x_3]  \in \wP_{(1, 2, 3, 5)}^3 (\Q)$ is normalized then it is unique, unless some of the coordinates are zero. For example the points $[0, 1, 0, 0]$ and $[0, -1, 0, 0]$ are equivalent and both normalized. 

\qed
\end{exa}

Thus, the weighted greatest common divisor gives us a very nice and efficient way to represent point in weighted moduli spaces via normalized points.  Such normalized points have as small coefficients as possible.  We define the \textbf{magnitude} or \textbf{naive height} of a point $\p \in \wP_w^n (k)$ as
\begin{equation} 
\size (\p) = \max \left\{ |x_0|_\infty^{\frac 1 {q_0}} , \dots ,  |x_n|_\infty^{\frac 1 {q_n}}      \right\} 
\end{equation}
where $x_0, \dots , x_i$ are the coordinates of the normalized point. 

\begin{lem}
Let $\w$ be a set of weights, $k$ a number field,  and $\wP_\w^n (k)$ a well-formed weighted projective space. Then the  function
\begin{equation}
 \size : \wP_\w^n (k) \to \R 
\end{equation}
is well defined. 
\end{lem}

\proof   Since $\wP_\w^n (k)$ is well-formed then from \cref{cor-1} for each point $\p \in \wP_\w^n (k)$  its normalization is unique. 
The rest follows. 

\qed

The above function provides a nice way to order points in $\wP_\w^n (k)$.  Moreover, each point in a well-formed space $\wP_\w^n (k)$ is now uniquely represented with "small" coefficients.  This idea, first suggested in \cite{mandili-sh} was explored in \cite{gen-2} and \cite{gen-3} to create a database and hyperelliptic curves of genus $g=2, 3$. 

Of course, the values of  $\size$   change as the field is extended.   We see an example below.

\begin{exa}
Let $\w=(2, 3, 5)$ and $\p=[7:0:0]\in \wP_\w^2 (\Q)$.  Then  $\wgcd_{\Q} (\p) =1$ and its normalization is $\bar \p = \p$.  Hence,   $ \size_\Q (\p) = \sqrt{7}$.

Consider now the field $K=\Q (\sqrt{7})$ and the same point $\p = [7:0:0] \in \wP_\w^2 (K)$.  Then $\wgcd_K (\p) = \sqrt{7}$ and the normalization of $\p$ is $\bar \p = [1:0:0]$.  Hence, $\size_K (\p) =1$.  

\qed

\end{exa}
 
So a different measuring of the size of points in $\wP_\w^n (k)$ is needed which behaves similarly to  a height function on the regular projective space $\P^n (k)$.  We explore this in the next section. 

\section{Heights on the weighted projective space}\label{sec-4}

In our attempt to define a height on the weighted projective space we fix the following notation for the rest of the paper. \\

$k$ \; is a number field.

$\O_k$ \; ring of integers of $k$.

$M_k$ \;  a complete set of absolute values of $k$

$M_k^0$ \;  the set of all non-archimedian places in $M_k$

$M_k^\infty$ \;  the set of archimedian places

$\X/k$ \;  a smooth projective variety defined over $k$. \\

Let $k$ be a given number field, $\O_k$ its ring of integers, and $M_k$ the set of absolute values on $k$.    For a place $\nu \in M_k$, the corresponding absolute value is denoted by $| \cdot |_\nu$, normalized with respect to $k$ such that the product formula holds and the Weil height is 
\[ 
H(x) =  \prod_\nu \max \{ 1, |x|_\nu \}. 
\]
For a point $\x \in k^{n+1}$ and a place $\nu \in M_k$ we define $| \x |_\nu = \max_i |x_i|_\nu$. For $\x = (x_0: \cdots : x_n)  \in \P^n (k)$ we have the height of $\x$ defined as 
\[ 
H(\x) = \prod_\nu \max \, \{ |x_0|_\nu, \ldots , |x_n|_\nu \} = \prod_\nu |x|_\nu 
\]
Because of the product formula, the height of $\x$ is well defined.

Let $k$ be an algebraic number field and $[k:\Q]=n$. With $M_k$ we will denote the set of all absolute values in $K$.   For $v \in M_k$, the \textbf{local degree at $v$}, denoted $n_v$ is 
\[n_v =[k_v:\Q_v]\]
where $K_v, \Q_v$ are the completions with respect to $v$. 
Let $L/k$ be an extension of number fields, and let $v \in M_k$ be an absolute value on $k$. Then
\[\sum_{\substack{w \in M_L\\w|v}}[L_w : k_v]= [L:k]\]
is known as the \textbf{degree formula}. For  $x \in k^\star$ we have the \textbf{product formula}  
\begin{equation}\label{prod.formula}
\prod_{v\in M_k}  |x|^{n_v}_v  =  1.   
\end{equation}
%
Given a point $\p \in \P^n(\overline \Q)$ with   $\p=[x_0, \dots, x_n]$, the \textbf{field of definition} of $\p$ is 
\[\Q(\p)=\Q   \left(  \frac  {x_0} {x_j}, \dots , \frac {x_n} {x_j}   \right)\] 
for any $j$ such that $x_j \neq 0$.  Next we try to generalize some of these concepts for the weighted projective spaces $\wP_{\w} (k)$.  

Let $\w=(q_0, \dots , q_n)$ be a set of heights and $\wP^n(k)$ the weighted  projective space  over  a number field $k$.   Let  $\p \in \wP^n(k)$ a point such that  $\p=[x_0, \dots , x_n]$.   Without any loss of generality we can assume that $\p$ is normalized.

 The \textbf{field of  absolute normalization } of $\p$ is defined as   $ \Q \left( \awgcd(\p)   \right)$. 
%

\begin{defi}\label{height}
Let $\w=(q_0, \dots , q_n)$ be a set of heights and $\wP^n (k)$ the weighted  projective space  over  a number field $k$.   Let  $\p \in \wP^n(k)$ a point such that  $\p=[x_0, \dots , x_n]$. We define the  \textbf{weighted multiplicative height} of $P$   as  
\begin{equation}\label{def:height}
\wh_k( \p ) := \prod_{v \in M_k} \max   \left\{   \frac{}{}   |x_0|_v^{\frac {n_v} {q_0}} , \dots, |x_n|_v^{\frac {n_v} {q_n}} \right\}
\end{equation}
The \textbf{logarithmic height} of the point $\p$ is defined as follows
\begin{equation}
\wh^\prime_k(\p) := \log \wh_k (\p)=   \sum_{v \in M_k}   \max_{0 \leq j \leq n}\left\{\frac{n_v}{q_j} \cdot  \log  |x_j|_v \right\}.
\end{equation}
\end{defi}
Next we will give some basic properties of heights functions.
\begin{prop}
Let $k$ be a number field and $\p\in \wP^n (k)$ with weights $w = (q_0, \dots, q_n)$. Then the following are true:

i) The height $\wh_k(\p)$ is well defined, in other words it does not depend on the choice of  coordinates of $\p$

ii) $\wh_k(\p) \geq 1$.
\end{prop}

\proof  
i)  Let $\p=[x_0, \dots , x_n] \in \wP^n(k)$. Since $\p$ is a point in the weighted projective space, any other choice of homogenous coordinates for $\p$ has the form $[\lambda^{q_0} x_0, \dots, \lambda^{q_n} x_n]$, where  $ \lambda \in k^*$. Then
\[
\begin{split}
\wh_k\left([\lambda^{q_0} x_0, \dots, \lambda^{q_n} x_n]\right) &= \prod_{v \in M_k} \max_{0 \leq i \leq n}  \left\{\frac{}{} |\lambda^{q_i}  x_i|_v^{n_v/q_i}\right\} \\
& =  \prod_{v \in M_k} |\lambda |_v^{n_v} \max_{0 \leq i \leq n}  \left\{\frac{}{} | x_i|_v^{n_v/q_i}\right\}\\
& =\left(   \prod_{v \in M_k} |\lambda |_v^{n_v}\right)\cdot \left(  \prod_{v \in M_k} \max_{0 \leq i \leq n}  \left\{\frac{}{} |x_i|_v^{n_v/q_i}\right\} \right)\\
\end{split} 
\]
Applying the product formula we have
\[\wh_K\left([\lambda^q_0 x_0, \dots, \lambda^q_n x_n]\right) = \prod_{v \in M_K} \max_{0 \leq i \leq n} \left\{\frac{}{} |x_i|_v^{n_v/q_i}\right\}= \wh_K(\p)\]
This completes the proof of the first part. 

ii) For every point $\p \in \wP^n (k)$ we can find a representative $\p^\prime$  of $\p$ with weighted homogenous coordinates such that one of the coordinates is 1. Assume,  that $\p= [ x_0: \ldots: x_i: \dots:  x_n]$ such that $x_i\neq 0$.   Then take $\p^\prime = \lambda\star \p$, where $\lambda = \left( \frac 1 {x_i} \right)^{\frac 1 {q_i}}$  and 
$ \p^\prime =   [ y_0: \ldots: 1: \dots:  y_n]$,   where 
\[ 
y_j =   x_j \cdot  x_i^{- \frac {q_j} {q_i}}
\]
for $j=0, \dots, n$ and $j \neq i$. The  the height is
\[\begin{split} 
\wh_k (\p^\prime) &=   \prod_{v \in M_k} \max\left\{\frac{}{}|x_0|_v^{n_v/q_0} , \dots, |x_n|_v^{n_v/q_n} \right\} \\
& =   \prod_{v \in M_k}  \max\left\{\frac{}{} 1, |y_0|_v^{n_v/q_0} , \dots, |y_n|_v^{n_v/q_n}\right\}.
\end{split}
 \] 
Hence,  every factor in the product is at least 1.  Therefore, $\wh_K(P) \geq 1$. 

\qed

Let us see an example.

\begin{exa}
Consider the set of weights $\w = (2, 3, 5)$ and the point $\p =[7: 0: 0] \in \wP_\w (\Q)$.  Then, $\wgcd_\Z (\p) = 1$  and 
\[ 
\begin{split}
\wh_\Q (\p) & =\max \left\{ \sqrt{|7|_7} , \sqrt[3]{|0|_7 } , \sqrt[5]{|0|_7  }  \right\} \cdot \max \left\{  \sqrt{|7|_\infty }, \sqrt[3]{|0|_\infty }, \sqrt[5]{|0|_\infty  }   \right\} \\
&  = \max \left\{ \sqrt{ \frac 1 7}, 1, 1 \right\}  \cdot \max \left\{  \sqrt{7}, 1, 1\right\}  = \sqrt{7}   
\end{split}  
\]
Let us now consider $K=\Q (\sqrt{7})$. Then,  $ \wgcd_{\O_K} (\p) = \sqrt{7}$    and  over $K$ we have 
\[  \p = \frac 1 {\sqrt{7} } \star [7:0:0] = [1:0:0], \]
so $\wh_K (\p) = 1$.  
\qed
\end{exa}


From the \cref{height} we see that $\size_k $ ca be defined as 
\begin{equation} 
\size_k (\p) = \prod_{M_k^\infty}     \max   \left\{   \frac{}{}   |x_0|_v^{\frac {n_v} {q_0}} , \dots, |x_n|_v^{\frac {n_v} {q_n}} \right\}
\end{equation}
and 
\begin{equation}
\wh_k ( \p ) = \size_k (\p) \cdot \prod_{v \in M_k^0} \max   \left\{   \frac{}{}   |x_0|_v^{\frac {n_v} {q_0}} , \dots, |x_n|_v^{\frac {n_v} {q_n}} \right\}
\end{equation}


Let $\w$, $k$ be as above and $\p \in \wP^n(k)$ such that $p=[x_0 : x_1 : \ldots : x_n]$. Denote by $K = k (\awgcd (\p))$.  Then,  over $K$, the weighted greatest common divisor is the same as the absolute greatest common divisor,
\[ \wgcd_K (\p) = \awgcd_K (\p).\]
Moreover, $[K:k] < \infty$ and we have the following. 

 
\begin{prop}\label{prop-5}
Let $\w$, $K=\Q(\awgcd (\p) )$,  and $\p \in \wP^n(K)$, say $\p=[x_0 : x_1 : \ldots : x_n]$.   
Then the following are true:

i)  If       $\p$ is normalized in $K$, then
\begin{equation}
\wh_K (\p)=  \wh_\infty (\p) = \max_{0 \leq i \leq n}\left\{\frac{}{}|x_i|^{{n_{\nu}}/q_i}_\infty \right\}.
\end{equation}
%

ii) If     $L/K$ is a finite extension,  then 
\begin{equation}
\wh_L(\p)=  \wh_K (\p)^{[L:K]}.
\end{equation}

\end{prop}

\begin{proof}   Let $\p=[x_0, \dots , x_n] \in \wP^n(K)$.  Then, $\p$ will have a representative $[y_0, \dots, y_n]$ such that $y_i \in \O_K$ for all $i=0, \ldots, n$ and $\wgcd (y_0, \dots, y_n)=1$.  With such representative for the coordinates of $\p$, the non-Archimedean absolute values give no contribution to the height, and we obtain
\[\wh_K (\p)=  \max_{0 \leq j \leq n}\left\{\frac{}{}|x_j|^{n_\nu /q_j}_\infty \right\}\]

ii) Let $L$ be a finite extension of $k$ and $M_L$ the corresponding  set of absolute values. Then,  
\[\begin{split}
\wh_L(\p) & = \prod_{w \in M_L} \max_{0 \leq i \leq n}\left\{\frac{}{} |x_i|_w^{n_w/q_i} \right\}  =  \prod_{v \in M_k} \prod_{\substack{w \in M_L\\  w|v}}   \max_{0 \leq i \leq n}\left\{\frac{}{} |x_i|_v^{n_w/q_i} \right\}, \; \;   \text{(since $x_i \in k$)} \\
& = \prod_{v \in M_k}   \max_{0 \leq i \leq n}\left\{\frac{}{} |x_i|_v^{\frac{n_v \cdot [L:K]}{q_i} }\right \},\qquad \text{(degree formula)}\\
&= \prod_{v \in M_k}   \max_{0 \leq i \leq n}\left\{\frac{}{} |x_i|_v^{ n_v/q_i} \right\}^{[L:k]}=\wh_k(\p)^{[L:k]}
\end{split} \]
This completes the proof. 
\end{proof}


\begin{cor}
If $\p$ is absolutely normalized over a number field $k$ then 
\[ \wh_k (\p) = \wh_\infty (\p) = \size_k (\p). \]
\end{cor}

The next example illustrates the previous Proposition.

\begin{exa}
Let $\w=(2,4)$ and $\p = [5\cdot 3, 5^2 \cdot 7]$. Then $\wgcd( 5\cdot 3, 5^2 \cdot 7  )=1$.  The height $\wh_\Q (\p)$ is
\[
\begin{split}
 \wh_\Q (\p) &= \max \left\{  \left(  \frac 1 5\right)^{\frac 1 2 }, \left(  \frac 1 {25} \right)^{\frac 1 4 }     \right\}  
\cdot 
\max \left\{  \left(  \frac 1 3\right)^{\frac 1 2 }, 1    \right\} 
\cdot
\max \left\{  1 , \left(  \frac 1 {7} \right)^{\frac 1 4 }     \right\}   \cdot  \left( 5 \cdot 3   \right)^{\frac 1 2 } \\
  & = \left(  \frac 1 5\right)^{\frac 1 2 } 
  \cdot \left( 5 \cdot 3   \right)^{\frac 1 2 } = \sqrt{3}.
\end{split}  
\]
The absolute weighted greatest common divisor is 
\[ \awgcd (  5\cdot 3, 5^2 \cdot 7 ) = \sqrt{5}.
\] 
Let $K=\Q (\sqrt{5})$ and compute $\wh_K (\p)$.  Over $K$ the point $\p$ is $\p = [ 3 :  7]$. Then
\[ 
\wh_K (\p) = \max \left\{ |3 |_\infty^{2/2},  |7|_\infty^{2/4}   \right\}  =  \max \{ 3, \sqrt{7} \} =3, 
\]
as expected from \cref{prop-5}, ii). 
\qed
\end{exa}

Let $q=q_0 q_1 \cdots q_n$ and consider the map 
\begin{equation}\label{map}
\begin{split}
\phi : \quad & \wP^n (k) \to \P^n (k) \\
& [x_0, \ldots , x_n] \to \left[  x_0^{\frac q {q_0}}, \ldots , x_n^{\frac q {q_n}}    \right] \\
\end{split}
\end{equation}

\begin{lem} \label{wh-H} 
Given $\phi$ and $q$ satisfying the above conditions we have 

i) $\phi$ is well-defined

ii) $\wh_\Q (\p) = H_\Q (\phi (\p) )^{ \frac 1 q }$

\end{lem}

\proof
Let $\x= [x_0, \dots , x_n]$ and $\y=[ y_0, \dots , y_n]$ be points in $\wP^n$ such that $\x = \lambda \star \y$.  Then, $\x= [ \lambda^{q_0} y_0, \ldots , \lambda^{q_n} y_n]$ and 
\[
\phi (\x) = \left[  \lambda^q \,  y_0^{\frac q {q_0} }, \ldots , \lambda^q \,  y_n^{ \frac q {q_n}  }    \right] = \phi (\y).
\]
Let $\p = [x_0, \dots , x_n]$ and $\bar \p = \phi(\p)$.  For the second part, by definition we have that
\[ \wh (\p)  =   \prod_{\nu\in M_k} \max \left\{  |x_i|_\nu^{\frac {n_v} {q_i}} \right\}  \]
and 
\[ H (\phi (\p)) =   \prod_{\nu\in M_k} \max \left\{  \left|x_i^{\frac{q}{q_i}}\right|_\nu^{n_v} \right\}.   \]
%
Then, we have the following
\[ \begin{split}
\left(\frac{}{} H (\phi (\p))\right)^{1/q} & = \left(   \prod_{\nu\in M_k} \max \left\{  \left|x_i^{\frac{q}{q_i}}\right|_\nu^{n_v} \right\} \right)^{1/q}
 = \prod_{\nu\in M_k} \max \left\{  \left|x_i^{\frac{1}{q_i}}\right|_\nu^{n_v} \right\}
 =  \wh (\p). 
 \end{split} \]

Therefore, we get $\wh_\Q (\p)^ q = H_\Q (\phi (\p) )$. 
\qed

\begin{cor}
The following holds for logarithmic heights
\[ q \log \wh (\p) = \log H (\phi (\p). \]
\end{cor}



Using \cref{prop-5}, part ii),  we can define the height on $\wP^n(\overline \Q)$. The height of a point on $\wP^n(\overline \Q)$ is called the \textbf{absolute (multiplicative) weighted height} and is the function 
\[
\begin{split}
\awh: \wP^n(\bar \Q) & \to [1, \infty)\\
\awh(\p)&=\wh_K(\p)^{1/[K:\Q]},
\end{split}
\]
where  $\p \in \wP^n(K)$, for any $K$ which contains $\Q (\awgcd (\p))$.  The \textbf{absolute (logarithmic) weighted height} on $\wP^n(\overline \Q)$  is the function 
\[
\begin{split}
\awh^\prime: \wP^n(\bar \Q) & \to [0, \infty)\\
\awh^\prime (\p)&= \log \, \wh (\p)= \frac 1 {[K:\Q]} \awh_K(\p).
\end{split}
\]

\begin{lem}\label{lem_galois_conj}
The height is invariant under Galois conjugation. In other words, for  $\p \in \wP^n(\overline \Q)$ and $\sigma \in G_{ \Q}$ we have $\wh (\p^\sigma) = \wh (\p)$. 
\end{lem}

\proof   Let $\p =[x_0, \dots, x_n] \in \wP^n(\overline \Q)$. Let $K$ be a finite Galois extension of $\Q$ such that $\p \in \wP^n(K)$. Let $\sigma \in G_\Q$. Then $\sigma$ gives an isomorphism 
\[\sigma: K \to K^\sigma\]
and also identifies the sets $M_K$, and $M_{K^\sigma}$ as follows
\[
\begin{split}
\sigma: M_K &\to M_{K^\sigma}\\
v &\to v^\sigma
\end{split}
\]
Hence, for every $x \in K$ and $v \in M_K$, we have $|x^\sigma|_{v^\sigma} = |x|_v$.   Obviously $\sigma$ gives as well an isomorphism 
\[\sigma: K_v \to K^\sigma_{v^\sigma}\]
Therefore $n_v= n_{v^\sigma}$, where $n_{v^\sigma} = [ K^\sigma_{v^\sigma}: \Q_v]$. Then 
\[
\begin{split}
\wh_{K^\sigma}(P^\sigma) &= \prod_{w \in M_{K^\sigma}} \max_{0 \leq i \leq n} \left\{\frac{}{}  |x_i^\sigma|_{w}^{n_{w}/q_i}\right\}\\
&=  \prod_{v \in M_{K}} \max_{0 \leq i \leq n} \left\{\frac{}{}  |x_i^\sigma|_{v^\sigma}^{n_{v^\sigma}/q_i}\right\}=  \prod_{v \in M_{K}} \max_{0 \leq i \leq n} \left\{\frac{}{}  |x_i|_v^{n_v/q_i}\right\}=\wh_K(\p)
\end{split}
\]
This completes the proof.
\qed

Given $\p \in \wP^n (K)$ as $\p = [x_0, \dots , x_n]$ the \textbf{field of definition} of $\p$ is defined as 
\[ \Q (\p) := \Q \left(      
\left( \frac {x_0} {x_i} \right)^{\frac {q_0} {q}}, \ldots , 1, \ldots ,   \left( \frac {x_n} {x_i} \right)^{\frac {q_n} {q}}
 \right)
\]
Notice that $\Q (\p)$ is the field containing all the liftings of the field $\Q (\phi (\p))$.  In other words adjoining al the $q$-roots to the minimal field of definition $\Q (\phi (\p))$ of $\phi (\p) \in \P^n$.


\begin{lem} 
For any point $\p \in \wP_\w^n (\overline \Q)$, we have 
\[ 
[ \Q (\p) : \Q] \leq  q \cdot [\Q(\phi (\p) ) : \Q ]
\]
\end{lem}

\proof
The proof follows from the fact that for every coordinate we have to possibly  adjoin at most a $q$-th root of unity. 

\qed

The following result is  analogue to  Northcott's theorem for weighted projective spaces.

\begin{thm} \label{thm_finite}
Let $c_0$ and $d_0$ be constants and $\wP_w^n(\overline \Q)$ the weighted projective space with weights  $\w = (q_0, \dots, q_n)$. Then the set 
\[
\{\p \in \wP_w^n(\overline \Q): \wh_\Q (\p) \leq c_0 \text{ and } [\Q(\p):\Q] \leq d_0\}
\]
contains only finitely many points. 
\end{thm} 

\proof The proof is a direct consequence of Northcott's theorem for projective spaces and \cref{wh-H}.  Let $q = q_0 q_1 \cdots q_n$ and consider the map 
$\phi :  \wP^n (k) \to \P^n (k)$ as defined in \cref{map}.   
From Northcott's theorem for projective spaces we have that if   $C_0=c_0^q$ and $D_0 = \frac 1 q d_0$ are  constants and $\P^n(\overline \Q)$ a projective space, then the set 
\[
\{ \phi (\p) \in \P^n(\overline \Q): H (\phi (\p )) \leq C_0 \text{ and } [ \Q (\phi (\p)) :\Q] \leq D_0\},
\]
contains only finitely many points  $\phi (\p)$.   
From  \cref{wh-H} we have that  
\[ 
\wh_\Q (\p) = H_\Q (\phi (\p) )^{ \frac 1 q } \leq C_0^{\frac 1 q} = c_0.
\] 
Also, 
\[ [ \Q (\p) : \Q] \leq  q \cdot [\Q(\phi (\p) ) : \Q ] \leq q \cdot D_0 = d_0. 
\]
Since $\phi$ is a finite degree map, then are only finitely many points $\p \in \wP_\w^n (\overline \Q)$ satisfying the above conditions.  This completes the proof. 

\qed

The following theorem is a more practical result especially from the computational point of view. 

\begin{cor}
There are finitely many absolutely normalized points $\p \in \wP_\w^n (\overline \Q)$ of bounded height.  In other words, 
%
\[\{\p \in \wP_w^n( \overline \Q): \wh_{\bar \Q} (\p) \leq c_0 \} \]
is a finite set for any constant $c_0$.

\end{cor}

\proof  Since $\p$ is absolutely normalized then $\wgcd (\p)=1$.  In this case $\Q (\p) = \Q$. The result follows from the above theorem.

\qed

\begin{cor}
For any number field $K$, the set 
\[ \{\p \in \wP_w^n(K):  \; \Q (\p) \subset K \; \text{ and } \;       \wh_K(\p) \leq c_0 \}, \]
is a finite set.
\end{cor}

\proof  Since $\Q (\p) \subset K $ then $\p$ is absolutely normalized $\wP_w^n(K)$. The result follows from the above. 

\qed

The next result is the analogue of what is called Kronecker's theorem for  heights on projective spaces.

\begin{lem}
Let $K$ be a number field,  and let $\p=[x_0: \dots: x_n] \in \wP^n_w(K)$, where $\w = (q_0, \dots, q_n)$. Fix any $i$ with $x_{i} \neq 0$. Then $\wh (\p)= 1$ if  the ratio $x_j/\xi_{i}^{q_j}$, where $\xi_i$ is the $q_i$-th root of unity of $x_i$, is a root of unity or zero for every $0 \leq j \leq n$ and $j \neq i$.
\end{lem}

\proof   Let  $\p=[x_0: \dots: x_i: \dots: x_n] \in \wP^n(K)$. Assume $x_i \neq 0$. Adjoin the $q_i$-th root of unity to $x_i$. Hence, let $x_i = \xi_i^{q_i}$  so that $wt(\xi_i) = 1$. Without loss of generality we can divide the coordinates of $\p$ by $\xi_i^{q_j}$, for $j \neq i$,  and then  we have 
$$\p=\left [\frac{x_0}{\xi_i^{q_0}}, \dots , 1, \dots, \frac{ x_n}{\xi_i^{q_n}}\right].$$ 
For simplicity let $\p=[y_0: \dots: 1: \dots: y_n]$.  If $y_l$   is a root of unity for every $0 \leq l \leq n$ and $l \neq i$ then $|y_l|_v=1$ for every $v \in M_K$. Hence, $\wh (\p)=1$.
\qed


\section{Concluding remarks}

The weighted greatest common divisors are a natural extension of the concept of greatest common divisors to weighted tuples.   Wether the usual properties of the greatest common divisors for Dedekind Domains can be extended to the weighted greatest common divisors is a natural question that needs further study. 
Even more generally  how the ideal calculus (\cite[Appendix A]{razon} ) can be generalized in terms of weighted ideals? For example, can Lemma~5 and Lemma~6 in \cite{razon} be generalized for weighted greatest common divisors? 

From the computational point of view it seems as there is no escape from the fact that to compute the weighted greatest common divisor one has to factor integers into primes. However, this is a problem that surely will be further investigated by computer algebra experts.  

The theory of heights is fundamental in arithmetic geometry and heights for weighted projective spaces provide powerful tools to study rational points in such spaces or on weighted Abelian varieties. The weighted projective height has the basic properties of the projective height. Whether this can be used to fully develop an arithmetic geometry machinery over weighted projective spaces remains to be seen.

\nocite{*}
\bibliographystyle{amsplain} 

\bibliography{mybib}{}

\begin{bibdiv}
\begin{biblist}

\bib{MR879909}{article}{
      author={Beltrametti, Mauro},
      author={Robbiano, Lorenzo},
       title={Introduction to the theory of weighted projective spaces},
        date={1986},
        ISSN={0723-0869},
     journal={Exposition. Math.},
      volume={4},
      number={2},
       pages={111\ndash 162},
      review={\MR{879909}},
}

\bib{gen-3}{incollection}{
      author={Beshaj, L.},
      author={Polak, M.},
       title={On hyperelliptic curves of genus 3},
        date={2019},
   booktitle={Algebraic curves and their applications},
      series={Contemp. Math.},
      volume={724},
   publisher={Amer. Math. Soc., Providence, RI},
       pages={161\ndash 173},
         url={https://doi.org/10.1090/conm/724/14589},
      review={\MR{3916739}},
}

\bib{gen-2}{incollection}{
      author={Beshaj, Lubjana},
      author={Guest, Scott},
       title={The weighted moduli space of binary sextics},
        date={2019},
   booktitle={Algebraic curves and their applications},
      series={Contemp. Math.},
      volume={724},
   publisher={Amer. Math. Soc., Providence, RI},
       pages={33\ndash 44},
         url={https://doi.org/10.1090/conm/724/14583},
      review={\MR{3916733}},
}

\bib{MR2852925}{article}{
      author={Bini, Gilberto},
       title={Quotients of hypersurfaces in weighted projective space},
        date={2011},
        ISSN={1615-715X},
     journal={Adv. Geom.},
      volume={11},
      number={4},
       pages={653\ndash 667},
         url={https://doi.org/10.1515/ADVGEOM.2011.029},
      review={\MR{2852925}},
}

\bib{MR2216774}{book}{
      author={Bombieri, Enrico},
      author={Gubler, Walter},
       title={Heights in {D}iophantine geometry},
      series={New Mathematical Monographs},
   publisher={Cambridge University Press, Cambridge},
        date={2006},
      volume={4},
        ISBN={978-0-521-84615-8; 0-521-84615-3},
         url={https://doi.org/10.1017/CBO9780511542879},
      review={\MR{2216774}},
}

\bib{MR627828}{article}{
      author={Buium, Alexandru},
       title={Weighted projective spaces as ample divisors},
        date={1981},
        ISSN={0035-3965},
     journal={Rev. Roumaine Math. Pures Appl.},
      volume={26},
      number={6},
       pages={833\ndash 842},
      review={\MR{627828}},
}

\bib{c-m-sh}{article}{
      author={Clingher, A.},
      author={Malmendier, A.},
      author={Shaska, T.},
       title={Configurations of 6 lines and string dualities},
        date={2019},
     journal={Communications in Mathematical Physics},
      volume={to appear},
}

\bib{deng}{article}{
      author={Deng, An-Wen},
       title={Rational points on weighted projective spaces},
        date={1998},
     journal={arXiv preprint math},
        ISSN={9812082/},
}

\bib{igor}{incollection}{
      author={Dolgachev, Igor},
       title={Weighted projective varieties},
        date={1982},
   booktitle={Group actions and vector fields ({V}ancouver, {B}.{C}., 1981)},
      series={Lecture Notes in Math.},
      volume={956},
   publisher={Springer, Berlin},
       pages={34\ndash 71},
         url={https://doi.org/10.1007/BFb0101508},
      review={\MR{704986}},
}

\bib{MR3916746}{incollection}{
      author={Frey, Gerhard},
      author={Shaska, Tony},
       title={Curves, {J}acobians, and cryptography},
        date={2019},
   booktitle={Algebraic curves and their applications},
      series={Contemp. Math.},
      volume={724},
   publisher={Amer. Math. Soc., Providence, RI},
       pages={279\ndash 344},
         url={https://doi.org/10.1090/conm/724/14596},
      review={\MR{3916746}},
}

\bib{MR1745599}{book}{
      author={Hindry, Marc},
      author={Silverman, Joseph~H.},
       title={Diophantine geometry},
      series={Graduate Texts in Mathematics},
   publisher={Springer-Verlag, New York},
        date={2000},
      volume={201},
        ISBN={0-387-98975-7; 0-387-98981-1},
         url={https://doi.org/10.1007/978-1-4612-1210-2},
        note={An introduction},
      review={\MR{1745599}},
}

\bib{MR0114819}{article}{
      author={Igusa, Jun-ichi},
       title={Arithmetic variety of moduli for genus two},
        date={1960},
        ISSN={0003-486X},
     journal={Ann. of Math. (2)},
      volume={72},
       pages={612\ndash 649},
         url={https://doi.org/10.2307/1970233},
      review={\MR{0114819}},
}

\bib{kaplansky}{book}{
      author={Kaplansky, Irving},
       title={Commutative rings},
   publisher={The University of Chicago Press},
        date={1974},
}

\bib{MR3712162}{article}{
      author={Malmendier, A.},
      author={Shaska, T.},
       title={The {S}atake sextic in {F}-theory},
        date={2017},
        ISSN={0393-0440},
     journal={J. Geom. Phys.},
      volume={120},
       pages={290\ndash 305},
         url={https://doi.org/10.1016/j.geomphys.2017.06.010},
      review={\MR{3712162}},
}

\bib{MR3731039}{article}{
      author={Malmendier, Andreas},
      author={Shaska, Tony},
       title={A universal genus-two curve from {S}iegel modular forms},
        date={2017},
        ISSN={1815-0659},
     journal={SIGMA Symmetry Integrability Geom. Methods Appl.},
      volume={13},
       pages={Paper No. 089, 17},
         url={https://doi.org/10.3842/SIGMA.2017.089},
      review={\MR{3731039}},
}

\bib{mandili-sh}{incollection}{
      author={Mandili, Jorgo},
      author={Shaska, Tony},
       title={Computing heights on weighted projective spaces},
        date={2019},
   booktitle={Algebraic curves and their applications},
      series={Contemp. Math.},
      volume={724},
   publisher={Amer. Math. Soc., Providence, RI},
       pages={149\ndash 160},
         url={https://doi.org/10.1090/conm/724/14588},
      review={\MR{3916738}},
}

\bib{razon}{article}{
      author={{Razon}, Aharon},
       title={{Primitive recursive decidability for large rings of algebraic
  integers.}},
    language={English},
        date={2019},
        ISSN={1930-1235/e},
     journal={{Albanian J. Math.}},
      volume={13},
      number={1},
       pages={1\ndash 91},
}

\bib{MR2403559}{article}{
      author={Rossi, Carlo~A.},
       title={Weighted projective spaces and minimal nilpotent orbits},
        date={2008},
        ISSN={1088-4165},
     journal={Represent. Theory},
      volume={12},
       pages={208\ndash 224},
         url={https://doi.org/10.1090/S1088-4165-08-00328-2},
      review={\MR{2403559}},
}

\bib{MR3525576}{incollection}{
      author={Shaska, T.},
      author={Beshaj, L.},
       title={Heights on algebraic curves},
        date={2015},
   booktitle={Advances on superelliptic curves and their applications},
      series={NATO Sci. Peace Secur. Ser. D Inf. Commun. Secur.},
      volume={41},
   publisher={IOS, Amsterdam},
       pages={137\ndash 175},
      review={\MR{3525576}},
}

\bib{MR2162351}{article}{
      author={Silverman, Joseph~H.},
       title={Generalized greatest common divisors, divisibility sequences, and
  {V}ojta's conjecture for blowups},
        date={2005},
        ISSN={0026-9255},
     journal={Monatsh. Math.},
      volume={145},
      number={4},
       pages={333\ndash 350},
         url={https://doi.org/10.1007/s00605-005-0299-y},
      review={\MR{2162351}},
}

\end{biblist}
\end{bibdiv}

\end{document}